\documentclass[12pt]{amsart}
\usepackage[utf8]{inputenc}
\usepackage[T1]{fontenc}
\usepackage{amsmath, amssymb, amsthm}
\usepackage{geometry}
\usepackage{booktabs}
\usepackage{color}
\usepackage{xcolor}
\usepackage{graphicx}
\usepackage{hyperref}
\usepackage{multicol}
\usepackage{listings}
\usepackage{caption}
\hypersetup{colorlinks=true,linkcolor=blue,citecolor=blue,urlcolor=blue}
\newtheorem{theorem}{Theorem}[section]

\theoremstyle{definition}

\newtheorem{lemma}{Lemma}[section]
\newtheorem{proposition}{Proposition}[section]
\newtheorem{corollary}{Corollary}[section]

\title[Weighted Product Inequalities for the Sine Function]{Weighted Product Inequalities for the Sine Function: A Gamma-Function Approach and Sharp Comparisons}
\author {Augustine L. Mahu}
\address{Department of Mathematics, University of Ghana, PO. Box LG 62 Legon, Accra, Ghana }
\email{ almahu@ug.edu.gh}
\author {Beno\^it F. Sehba}
\address{Department of Mathematics, University of Ghana, PO. Box LG 62 Legon, Accra, Ghana }
\email{ bfsehba@ug.edu.gh}
\author {Cecilia D. Williams}
\address{Department of Mathematics, University of Ghana, PO. Box LG 62 Legon, Accra, Ghana }
\email{ cdwilliams@st.ug.edu.gh}

\date{}

\begin{document}

\maketitle

\begin{abstract}
Using the log-convexity of the Gamma function and Euler's reflection formula, we give a new proof of a classical weighted sine product inequality. Two different parameter choices yield two competing upper bounds for the same product. We determine precisely, via algebraic criteria, when one bound is sharper than the other. Explicit results are given for the general $n$-angle case, the $2n$-angle case, and for two and three angles. Several sharp corollaries are derived, including $\sin(\pi x)\leq \sin(2\pi x(1-x))$.
\end{abstract}

\bigskip

\section{Introduction}


Inequalities for the sine and cosine have attracted continuous
attention since the early development of mathematical analysis.
The oldest and most cited is Jordan's inequality
\cite{Jordan1894},
\begin{equation}\label{eq:Jordan}
  \frac{2}{\pi}\,\theta \;\leq\; \sin\theta \;\leq\; \theta,
  \qquad 0\leq\theta\leq\frac\pi2,
\end{equation}
which encodes the global concavity of sine on $[0,\pi/2]$. This inequality has been refined and extended by various authors (see \cite{Qi2009,Zhu2006}). 

A second classical line of research concerns products of values
of the sine function.  Wallis's infinite product (1655)
\cite{Wallis1656},
\begin{equation}\label{eq:Wallis}
  \frac\pi2 = \prod_{k=1}^\infty\frac{(2k)(2k)}{(2k-1)(2k+1)},
\end{equation}
can be rewritten in terms of $\sin(k\pi/n)$ and motivated a
long study of inequalities for such products.
The monograph of Mitrinovi\'c \cite{Mitrinovic1970} and the handbook
of Bullen \cite{Bullen2003} give extensive accounts of the classical
theory of trigonometric and power-mean inequalities, including
many results on products of sine values.

We observe that the function $\sin$ is log-concave on $(0,\pi)$. This fact can be combined with the Jensen's inequality to show the following: 
for positive weights $\alpha_i$, $i=1,2,\cdots,n$, and $\alpha=\sum\alpha_i$,
\begin{equation}\label{eq:prodineqsine}
  \prod_{i=1}^{n}\sin^{\alpha_i}(\theta_i)
  \;\leq\;
  \sin^{\alpha}\!\left(\sum_{i=1}^n\frac{\alpha_i}{\alpha}\,\theta_i\right).
\end{equation}
Inequality \eqref{eq:prodineqsine} is the starting point of this
paper. For a proof and some variations, we refer the reader to
\cite{AVV1997,Bullen2003,CP2014,Mitrinovic1970,MPF1993,NS2010}. 

In this paper, we derive \eqref{eq:prodineqsine} not from the log-concavity of $\sin$ directly, but from the \emph{log-convexity} of the Gamma function and Euler's reflection formula. We then, through an appropriate choice of the weights and the angles, obtain two upper bounds in \eqref{eq:prodineqsine} for the same left hand side, for an even number of entries and for the general case. We compare these two upper bounds and obtain exact conditions under which one is better than the other. In the case of two or three weights, we give a more simple representation of these conditions.

While inequality \eqref{eq:prodineqsine} is classical, the following natural question does not appear to have been addressed in the literature: the right-hand side of \eqref{eq:prodineqsine} depends on the choice of weights $\alpha_i$, and for a given product on the left-hand side, different admissible weight choices can yield two distinct upper bounds. Which of these bounds is sharper, and under what exact conditions on the parameters? This paper provides a first and explicit answer.

The paper is organized as follows: Section~\ref{sec:prelim} collects background on the Beta and Gamma functions, the reflection formula, and the log-convexity lemma. In Section~\ref{sec:main}, we present and prove our main results on the estimate \eqref{eq:prodineqsine} and its variations that provide the different upper bounds for the same left hand side. The comparison between the upper bounds is carried out in Section~\ref{sec:compare}. We provide a conclusion for this work in  Section ~\ref{sec:conclusion} where we also present some open problems.

Throughout, for positive quantities
$A$ and $B$, the notation $A\lesssim B$ means $A\leq CB$ for some
absolute constant $C>0$, and $A\approx B$ means $A\lesssim B$
and $B\lesssim A$.

\section{Preliminaries}\label{sec:prelim}

Recall that the Beta function is defined for $x>0$ and $y>0$ by
\[
  B(x,y):=\int_0^1 t^{x-1}(1-t)^{y-1}\,dt.
\]
A related function is the Euler's Gamma function defined by
\[
  \Gamma(x)=\int_0^\infty e^{-t}\,t^{x-1}\,dt,\qquad x>0.
\]
In effect, we have  $$B(x,y)=\frac{\Gamma(x)\Gamma(y)}{\Gamma(x+y)}.$$  For any $0<x<1$,
the following known as the \emph{reflection formula} holds:
\begin{equation}\label{eq:reflection}
  B(1-x,x)=\Gamma(1-x)\,\Gamma(x)=\frac{\pi}{\sin(\pi x)}.
\end{equation}

This formula transforms inequalities for $1/\sin(\pi x)$ into inequalities
for products of Gamma values and vice versa.

The Gamma function is log-convex on $(0,\infty)$.  By the
Bohr--Mollerup theorem \cite{BohrMollerup1922}, $\Gamma$ is the
only log-convex function satisfying $\Gamma(x+1)=x\Gamma(x)$
and $\Gamma(1)=1$.  The quantitative form needed here is as follows.

\begin{lemma}\label{lem:logconvex}
Let $1<p_i<\infty$, $i=1,2,\ldots,n$, and define $p$ by
$\dfrac{1}{p}=\displaystyle\sum_{i=1}^{n}\dfrac{1}{p_i}$.
If $a_1,\ldots,a_n\in[0,1]$, then
\begin{equation}\label{eq:logconvex}
  \Gamma\!\left(\sum_{i=1}^n\frac{p\,a_i}{p_i}\right)
  \;\leq\;
  \prod_{i=1}^n\Gamma^{p/p_i}(a_i).
\end{equation}
\end{lemma}

\begin{proof}
This is a consequence of the integral representation of $\Gamma$
and the generalized H\"older's inequality.
\end{proof}

\section{Main results}\label{sec:main}
In this section, we present our approach to the inequality \eqref{eq:prodineqsine} through the log-convexity of the Gamma function. We also give for the same left hand side in \eqref{eq:prodineqsine}, two upper bounds. Let us start with the following result.
\begin{theorem}\label{thm:main}
Let $1<p_i<\infty$, $i=1,2,\ldots,n$, and define $p$ by
$\dfrac{1}{p}=\displaystyle\sum_{i=1}^{n}\dfrac{1}{p_i}$.
If $a_1,\ldots,a_n\in[0,1]$, then
\begin{equation}\label{eq:main}
  \Gamma\!\left(\sum_{i=1}^n\frac{p\,a_i}{p_i}\right)
  \Gamma\!\left(\sum_{i=1}^n\frac{p(1-a_i)}{p_i}\right)
  \;\leq\;
  \prod_{i=1}^n
  \bigl[\Gamma(a_i)\,\Gamma(1-a_i)\bigr]^{p/p_i}.
\end{equation}
\end{theorem}

\begin{proof}
By Lemma~\ref{lem:logconvex} applied to $(a_i)$ and to $(1-a_i)$,
\begin{equation}\label{eq:logconvex1a}
    \Gamma\!\left(\sum_{i=1}^n\frac{p\,a_i}{p_i}\right)
  \leq \prod_{i=1}^n\Gamma^{p/p_i}(a_i),
\end{equation}
and
\begin{equation} \label{eq:logconvex1b}
     \Gamma\!\left(\sum_{i=1}^n\frac{p(1-a_i)}{p_i}\right)
  \leq \prod_{i=1}^n\Gamma^{p/p_i}(1-a_i).
\end{equation}
Multiplying \eqref{eq:logconvex1a} and \eqref{eq:logconvex1b}
side by side yields \eqref{eq:main}.
\end{proof}

Applying the reflection formula \eqref{eq:reflection} to both
sides of \eqref{eq:main} now converts the Gamma inequality into
the desired sine inequality.

\begin{corollary}\label{cor:main}
Let $1<p_i<\infty$, $i=1,2,\ldots,n$, and define $p$ as above.
If $a_1,\ldots,a_n\in(0,1)$, then
\begin{equation}\label{eq:cormain}
  \prod_{i=1}^n\sin^{1/p_i}(\pi a_i)
  \;\leq\;
  \sin^{1/p}\!\left(\pi\sum_{i=1}^n\frac{p}{p_i}\,a_i\right).
\end{equation}
\end{corollary}

Note that \eqref{eq:cormain} is exactly \eqref{eq:prodineqsine}
with $\theta_i=\pi a_i$ and $\alpha_i=1/p_i$. Thus,
Corollary~\ref{cor:main} provides an alternative proof of the
classical inequality \eqref{eq:prodineqsine}.

We now give two formulations of Corollary~\ref{cor:main} when the number of parameters $n$ is even, by making two different choices of the auxiliary arguments $a_i$ in \eqref{eq:cormain}.

\begin{proposition}\label{prop:main1}
Let $n=2k$, $k\in\mathbb{N}$, and let
$\lambda_1,\ldots,\lambda_{2k}\in(0,1)$.
Set $\lambda=\displaystyle\sum_{i=1}^k
\bigl[(1-\lambda_{2i})+\lambda_{2i-1}\bigr]$.  Then
\begin{equation}\label{eq:propmain1}
  \prod_{i=1}^{k}\sin^{1-\lambda_{2i}}(\pi\lambda_{2i-1})\,
  \sin^{\lambda_{2i-1}}(\pi\lambda_{2i})
  \;\leq\;
  \sin^{\lambda}\!\left(\pi\sum_{i=1}^k
  \frac{\lambda_{2i-1}}{\lambda}\right).
\end{equation}
\end{proposition}

\begin{proof}
One checks that
\[
  \sin^\lambda\!\left(\pi\sum_{i=1}^k\frac{\lambda_{2i-1}}{\lambda}\right)
  = \sin^\lambda\!\left(\pi\,
  \frac{\sum_{i=1}^k
    \bigl[\lambda_{2i-1}(1-\lambda_{2i})+\lambda_{2i-1}\lambda_{2i}\bigr]}
  {\sum_{i=1}^k\bigl[(1-\lambda_{2i})+\lambda_{2i-1}\bigr]}\right).
\]
Then \eqref{eq:propmain1} follows from \eqref{eq:cormain} with
$a_{2i-1}=\lambda_{2i-1}$, $1/p_{2i-1}=1-\lambda_{2i}$,
$a_{2i}=\lambda_{2i}$, $1/p_{2i}=\lambda_{2i-1}$ and $\frac{1}{p}=\lambda$.
\end{proof}

The following second proposition produces a different weighted
mean angle, and hence a genuinely different upper bound.

\begin{proposition}\label{prop:main2}
Under the hypotheses of Proposition~\ref{prop:main1},
\begin{equation}\label{eq:propmain2}
  \prod_{i=1}^{k}\sin^{1-\lambda_{2i}}(\pi\lambda_{2i-1})\,
  \sin^{\lambda_{2i-1}}(\pi\lambda_{2i})
  \;\leq\;
  \sin^{\lambda}\!\left(\pi\sum_{i=1}^k
  \frac{2\lambda_{2i-1}(1-\lambda_{2i})}{\lambda}\right).
\end{equation}
\end{proposition}

\begin{proof}
Similarly, the right-hand side is
equal to $\sin^\lambda\!\left(\pi\,
\displaystyle\sum_{i=1}^k \frac{2\lambda_{2i-1}(1-\lambda_{2i})}{\lambda}\right)$.
Take $a_{2i-1}=\lambda_{2i-1}$, $1/p_{2i-1}=1-\lambda_{2i}$,
$a_{2i}=1-\lambda_{2i}$, $1/p_{2i}=\lambda_{2i-1}$
in \eqref{eq:cormain}, and use
$\sin(\pi(1-\lambda_{2i}))=\sin(\pi\lambda_{2i})$, and $\frac{1}{p}=\lambda$.
\end{proof}

The next two propositions cover the general case of $n$ parameters without the parity restriction.  

\begin{proposition}\label{prop:main3}
Let $\lambda_1,\ldots,\lambda_n\in(0,1)$ and
$\lambda=\displaystyle\sum_{i=1}^n\lambda_i$.  Then
\begin{equation}\label{eq:propmain3}
  \Bigl(\prod_{i=1}^{n-1}\sin^{\lambda_{i+1}}(\pi\lambda_i)\Bigr)
  \sin^{\lambda_1}(\pi\lambda_n)
  \;\leq\;
  \sin^\lambda\!\left(\frac\pi\lambda
  \Bigl[\sum_{i=1}^{n-1}\lambda_{i+1}(1-\lambda_i)+\lambda_1\lambda_n
  \Bigr]\right).
\end{equation}
\end{proposition}

\begin{proof}
Apply \eqref{eq:cormain} with $a_i=1-\lambda_i$,
$1/p_i=\lambda_{i+1}$ ($i=1,\ldots,n-1$), and
$a_n=\lambda_n$, $1/p_n=\lambda_1$, and $\frac{1}{p}=\lambda$.
\end{proof}

\begin{proposition}\label{prop:main4}
Under the hypotheses of Proposition~\ref{prop:main3},
\begin{equation}\label{eq:propmain4}
  \Bigl(\prod_{i=1}^{n-1}\sin^{\lambda_{i+1}}(\pi\lambda_i)\Bigr)
  \sin^{\lambda_1}(\pi\lambda_n)
  \;\leq\;
  \sin^\lambda\!\left(\frac\pi\lambda
  \Bigl[\sum_{i=1}^{n-1}\lambda_{i+1}\lambda_i+\lambda_1\lambda_n
  \Bigr]\right).
\end{equation}
\end{proposition}

\begin{proof}
Apply \eqref{eq:cormain} with $a_i=\lambda_i$,
$1/p_i=\lambda_{i+1}$ ($i=1,\ldots,n-1$), and
$a_n=\lambda_n$, $1/p_n=\lambda_1$, with $\frac{1}{p}=\lambda$.
\end{proof}

In the two cases above, both associated propositions share the same left-hand side. The comparison between their two right-hand sides is carried out in the next section. We also note that the two propositions can be used to generate some other variations on these products. We give an illustration that is not exhaustive in the following.
\begin{corollary}\label{cor:main34}
 Let $\lambda_1,\ldots,\lambda_n\in(0,1)$ and
$\lambda=\left(\displaystyle\sum_{i=1}^{n-1}\lambda_i\right)+2\lambda_n$.  Then
\begin{equation}\label{eq:propmain31}
  \Bigl(\prod_{i=1}^{n-1}\sin^{\lambda_{i+1}}(\pi\lambda_i)\Bigr)
  \sin^{\lambda_1+\lambda_n}(\pi\lambda_n)
  \;\leq\;
  \sin^\lambda\!\left(\frac\pi\lambda
  \Bigl[\sum_{i=1}^{n-1}\lambda_{i+1}(1-\lambda_i)+\lambda_n(\lambda_1+1-\lambda_n)
  \Bigr]\right)
\end{equation}
 and  

 \begin{equation}\label{eq:propmain41}
  \Bigl(\prod_{i=1}^{n-1}\sin^{\lambda_{i+1}}(\pi\lambda_i)\Bigr)
  \sin^{\lambda_1+\lambda_n}(\pi\lambda_n)
  \;\leq\;
  \sin^\lambda\!\left(\frac\pi\lambda
  \Bigl[\sum_{i=1}^{n-1}\lambda_{i+1}\lambda_i+\lambda_n(\lambda_1+\lambda_n).
  \Bigr]\right)
\end{equation}

\end{corollary}
\begin{proof}
Write Proposition \ref{prop:main3} and Proposition \ref{prop:main4} for $\lambda_1,\cdots,\lambda_{n+1}\in (0,1)$, then make $\lambda_{n+1}=\lambda_n$.    
\end{proof}

\section{Comparing the upper bounds}\label{sec:compare}

The aim of this section is to compare, for both families of inequalities,
the two upper bounds produced by the parameter choices of
Section~\ref{sec:main}.  

The following elementary observation is the key ingredient behind all the comparison results in this section.

\begin{lemma}\label{lem:sincompare}
Let $\lambda>0$ and $\sigma,\tau\in(0,\lambda)$.  Then
\[
  \sin\!\Bigl(\frac{\pi\sigma}{\lambda}\Bigr)
  \;\geq\;
  \sin\!\Bigl(\frac{\pi\tau}{\lambda}\Bigr)
  \quad\Longleftrightarrow\quad
  (\sigma-\tau)(\sigma+\tau-\lambda)\;\leq\;0,
\]
with equality if and only if $\sigma=\tau$ or $\sigma+\tau=\lambda$.
\end{lemma}

\begin{proof}
Both arguments $\pi\sigma/\lambda$ and $\pi\tau/\lambda$ lie in
$(0,\pi)$, on which $\sin$ achieves its maximum at $\pi/2$ and is
symmetric about that point, with $\sin\alpha=\sin(\pi-\alpha)$.
Thus, $\sin(\pi\sigma/\lambda)\geq\sin(\pi\tau/\lambda)$ if and only
if $\pi\sigma/\lambda$ is at least as close to $\pi/2$ as
$\pi\tau/\lambda$, i.e.\
\[
  \Bigl|\frac{\pi\sigma}{\lambda}-\frac{\pi}{2}\Bigr|
  \;\leq\;
  \Bigl|\frac{\pi\tau}{\lambda}-\frac{\pi}{2}\Bigr|.
\]
Squaring and multiplying through by $(\lambda/\pi)^2>0$, this is
$(\sigma-\lambda/2)^2\leq(\tau-\lambda/2)^2$, i.e.\
\[
  \bigl[(\sigma-\tfrac\lambda2)-(\tau-\tfrac\lambda2)\bigr]
  \bigl[(\sigma-\tfrac\lambda2)+(\tau-\tfrac\lambda2)\bigr]\leq 0,
\]
which simplifies to $(\sigma-\tau)(\sigma+\tau-\lambda)\leq 0$.
The equality holds if and only if $\sigma=\tau$ or $\sigma+\tau=\lambda$. The proof is complete.
\end{proof}

\subsection{The general \texorpdfstring{$n$}{n}-angle case}\label{sec:nangle}

Recall from Propositions~\ref{prop:main3} and~\ref{prop:main4} that
for parameters $\lambda_1,\ldots,\lambda_n\in(0,1)$ the two upper
bounds for the cyclic product
$\bigl(\displaystyle\prod_{i=1}^{n-1}\sin^{\lambda_{i+1}}(\pi\lambda_i)\bigr)
\sin^{\lambda_1}(\pi\lambda_n)$
have arguments (numerators, with common denominator $\lambda:=\sum_i\lambda_i$)
\begin{align*}
  \sigma &:= \sum_{i=1}^{n-1}\lambda_{i+1}(1-\lambda_i)+\lambda_1\lambda_n,\\
  \tau   &:= \sum_{i=1}^{n-1}\lambda_{i+1}\lambda_i +\lambda_1\lambda_n.
\end{align*}

\begin{theorem}\label{thm:compare_n}
Let $\lambda_1,\ldots,\lambda_n\in(0,1)$, $\lambda=\sum_i\lambda_i$,
and let $S$, $T$ denote the upper bounds of
Propositions~\ref{prop:main3} and~\ref{prop:main4} respectively.
Define
\begin{equation}\label{eq:delta_n}
  \delta_n \;:=\; \sum_{i=1}^{n-1}\lambda_{i+1}(1-2\lambda_i).
\end{equation}
Then
\begin{equation}\label{eq:compare_n}
  S \;\geq\; T
  \quad\Longleftrightarrow\quad
  \delta_n\,(2\lambda_n-1) \;\leq\; 0,
\end{equation}
with equality $S=T$ if and only if $\delta_n=0$ or $\lambda_n=\tfrac12$.
\end{theorem}

\begin{proof}
Since $t\mapsto t^\lambda$ is increasing for $t>0$, $S\geq T$ if and only if
$\sin(\pi\sigma/\lambda)\geq\sin(\pi\tau/\lambda)$.
By Lemma~\ref{lem:sincompare}, this is equivalent to
$(\sigma-\tau)(\sigma+\tau-\lambda)\leq 0$.

\medskip
Let us compute $\sigma-\tau$ and $\sigma+\tau-\lambda$. We have
\begin{equation}\label{eq:sigmatau_n}
  \sigma-\tau
  =\sum_{i=1}^{n-1}\lambda_{i+1}(1-\lambda_i)-\sum_{i=1}^{n-1}\lambda_{i+1}\lambda_i
  =\sum_{i=1}^{n-1}\lambda_{i+1}(1-2\lambda_i)
  =\delta_n.
\end{equation}

\medskip
We also obtain
\begin{align*}
  \sigma+\tau
  &=\sum_{i=1}^{n-1}\lambda_{i+1}(1-\lambda_i)+\sum_{i=1}^{n-1}\lambda_{i+1}\lambda_i
   +2\lambda_1\lambda_n
  =\sum_{i=1}^{n-1}\lambda_{i+1}+2\lambda_1\lambda_n\\
  &=(\lambda-\lambda_1)+2\lambda_1\lambda_n,
\end{align*}
since $\sum_{i=2}^n\lambda_i=\lambda-\lambda_1$.  Therefore,
\begin{equation}\label{eq:sigmasum_n}
  \sigma+\tau-\lambda = \lambda_1(2\lambda_n-1).
\end{equation}

\medskip
Substituting \eqref{eq:sigmatau_n} and \eqref{eq:sigmasum_n} into
the criterion of Lemma~\ref{lem:sincompare} gives
\[
  \delta_n\cdot\lambda_1(2\lambda_n-1)\leq 0.
\]
Since $\lambda_1>0$, this is equivalent to $\delta_n(2\lambda_n-1)\leq 0$,
which is \eqref{eq:compare_n}.

The equality $S=T$ holds if and only if $\delta_n=0$ (i.e. $\sigma=\tau$) or
$2\lambda_n-1=0$, i.e. $\lambda_n=\tfrac12$.  

The proof is complete.
\end{proof}

\subsection{The general \texorpdfstring{$2n$}{2n}-angle case}\label{sec:2nangle}

Recall from Propositions~\ref{prop:main1} and~\ref{prop:main2} that
for $\lambda_1,\ldots,\lambda_{2n}\in(0,1)$ the two bounds for the
paired product
$\displaystyle\prod_{i=1}^{n}\sin^{1-\lambda_{2i}}(\pi\lambda_{2i-1})
\sin^{\lambda_{2i-1}}(\pi\lambda_{2i})$
have total weight $\lambda=\displaystyle\sum_{i=1}^n[(1-\lambda_{2i})+\lambda_{2i-1}]$
and arguments
\begin{align*}
  \sigma_P &:= \sum_{i=1}^{n}\lambda_{2i-1},\\
  \sigma_Q &:= 2\sum_{i=1}^{n}\lambda_{2i-1}(1-\lambda_{2i}).
\end{align*}

\begin{theorem}\label{thm:compare_2n}
Let $\lambda_1,\ldots,\lambda_{2n}\in(0,1)$, and let $P$, $Q$ denote
the bounds of Propositions~\ref{prop:main1} and~\ref{prop:main2}
respectively.  Define
\begin{equation}\label{eq:Delta_n}
  \Delta_n := \sum_{i=1}^{n}\lambda_{2i-1}(2\lambda_{2i}-1),
  \qquad
  \Sigma_n := \sum_{i=1}^{n}(1-\lambda_{2i})(2\lambda_{2i-1}-1).
\end{equation}
Then
\begin{equation}\label{eq:compare_2n}
  P \;\geq\; Q
  \quad\Longleftrightarrow\quad
  \Delta_n\,\Sigma_n \;\leq\; 0,
\end{equation}
with equality $P=Q$ if and only if $\Delta_n=0$ or $\Sigma_n=0$.
\end{theorem}

\begin{proof}
Again $P\geq Q$ if and only if $\sin(\pi\sigma_P/\lambda)\geq\sin(\pi\sigma_Q/\lambda)$,
which by Lemma~\ref{lem:sincompare} is
$(\sigma_P-\sigma_Q)(\sigma_P+\sigma_Q-\lambda)\leq 0$.

\medskip
\noindent\textbf{Computing $\sigma_P-\sigma_Q$.}
\begin{align*}
  \sigma_P-\sigma_Q
  &=\sum_{i=1}^n\lambda_{2i-1}-2\sum_{i=1}^n\lambda_{2i-1}(1-\lambda_{2i})
  =\sum_{i=1}^n\lambda_{2i-1}\bigl[1-2(1-\lambda_{2i})\bigr]\\
  &=\sum_{i=1}^n\lambda_{2i-1}(2\lambda_{2i}-1)
  =\Delta_n.
\end{align*}

\medskip
\noindent\textbf{Computing $\sigma_P+\sigma_Q-\lambda$.}
Write $S_{\mathrm{odd}}:=\sum_{i=1}^n\lambda_{2i-1}$ and
$S_{\mathrm{even}}:=\sum_{i=1}^n\lambda_{2i}$, so that
$\lambda=n-S_{\mathrm{even}}+S_{\mathrm{odd}}$.  Then
$\sigma_P+\sigma_Q=S_{\mathrm{odd}}+2\sum_i\lambda_{2i-1}(1-\lambda_{2i})
=3S_{\mathrm{odd}}-2\sum_i\lambda_{2i-1}\lambda_{2i}$, and
\begin{align*}
  \sigma_P+\sigma_Q-\lambda
  &=3S_{\mathrm{odd}}-2\sum_i\lambda_{2i-1}\lambda_{2i}
    -(n-S_{\mathrm{even}}+S_{\mathrm{odd}})\\
  &=\sum_{i=1}^n\bigl[2\lambda_{2i-1}+\lambda_{2i}-1
    -2\lambda_{2i-1}\lambda_{2i}\bigr]\\
  &=\sum_{i=1}^n\bigl[2\lambda_{2i-1}(1-\lambda_{2i})-(1-\lambda_{2i})\bigr]\\
  &=\sum_{i=1}^n(1-\lambda_{2i})(2\lambda_{2i-1}-1)
  =\Sigma_n.
\end{align*}

\medskip
Thus, the criterion $(\sigma_P-\sigma_Q)(\sigma_P+\sigma_Q-\lambda)\leq 0$
becomes $\Delta_n\,\Sigma_n\leq 0$, which is \eqref{eq:compare_2n}. The proof is complete.
\end{proof}

\subsection{The two-dimensional case}\label{sec:2D}

We now restrict Theorem~\ref{thm:compare_n} to the case $n=2$.
The following is obtained by combining
Propositions~\ref{prop:main1}--\ref{prop:main2} (or equivalently
Propositions~\ref{prop:main3}--\ref{prop:main4}) for $n=2$.

\begin{corollary}\label{cor:main12}
For $0<x<1$ and $0<y<1$, define
\[
  A(x,y):=\Bigl(\sin\tfrac{\pi y}{1-x+y}\Bigr)^{\!1-x+y},
  \qquad
  B(x,y):=\Bigl(\sin\tfrac{2\pi y(1-x)}{1-x+y}\Bigr)^{\!1-x+y}.
\]
Then
\begin{equation}\label{eq:main12}
  \bigl(\sin(\pi x)\bigr)^y\,\bigl(\sin(\pi y)\bigr)^{1-x}
  \;\leq\;\min\{A(x,y),\,B(x,y)\}.
\end{equation}
\end{corollary}

Setting $x=1-ky$, $k\in\mathbb{N}$, the above reduces to the following.

\begin{corollary}\label{cor:main121}
For $0<y<\frac{1}{k}$ with $k\in\mathbb{N}$, we have
\begin{equation}\label{eq:cormain121}
  \sin(k\pi y)\,\sin^k(\pi y)
  \;\leq\; \begin{cases}
   \sin^{k+1}\!\bigl(\tfrac {\pi }{k+1}\bigr) &  \text{if }\; y\in (0,\frac{1}{2k}],\\[4pt]
        \sin^{k+1}\!\bigl(\tfrac{2\pi ky}{k+1}\bigr) & \text{if }\;y\in [\frac{1}{2k},\frac{1}{k}).
 \end{cases}
\end{equation}
\end{corollary}

A particular case worth noting is
\begin{equation}\label{eq:ysinineq}
  \sin(2\pi y)\,\sin^2(\pi y)\leq
  \tfrac{3\sqrt3}{8}, \qquad 0< y<\frac{1}{4}.
\end{equation}

The comparison between $A$ and $B$ is the instance $n=2$ of
Theorem~\ref{thm:compare_n}: with $\lambda_1=x$, $\lambda_2=y$,
one has $\delta_2=\lambda_2(1-2\lambda_1)=y(1-2x)$ and
$2\lambda_2-1=2y-1$.  For the $2n$ family with $n=1$, Theorem~\ref{thm:compare_2n}
gives the same result via $\Delta_1=x(2y-1)$ and $\Sigma_1=(1-y)(2x-1)$.
We record the explicit statement.

\begin{corollary}\label{cor:compare2D}
For $0<x,y<1$, the following holds.
\[
  \begin{cases}
    A(x,y) < B(x,y) & \text{if }\;(x-\tfrac12)(y-\tfrac12)<0,\\[4pt]
    A(x,y) = B(x,y) & \text{if }\;x=\tfrac12\text{ or }y=\tfrac12,\\[4pt]
    A(x,y) > B(x,y) & \text{if }\;(x-\tfrac12)(y-\tfrac12)>0.
  \end{cases}
\]
\end{corollary}

\begin{proof}
With $\lambda_1=x$, $\lambda_2=y$, and $n=2$, Theorem~\ref{thm:compare_n}
gives $S\geq T$ if and only if $\delta_2(2\lambda_2-1)\leq 0$, where
$\delta_2=y(1-2x)$.  This becomes $y(1-2x)(2y-1)\leq 0$.
Since $y>0$, this reduces to $(2x-1)(2y-1)\geq 0$, that is
$(x-\tfrac12)(y-\tfrac12)\geq 0$.  Here $S=B$ and $T=A$ in the
notation of Corollary~\ref{cor:main12}, so $B\geq A$ if and only if
$(x-\tfrac12)(y-\tfrac12)\geq 0$, which is the stated result. The proof is complete.
\end{proof}

\begin{multicols}{2} 
    \noindent 
    \begin{minipage}{\columnwidth}
        \centering
        \includegraphics[width=\linewidth]{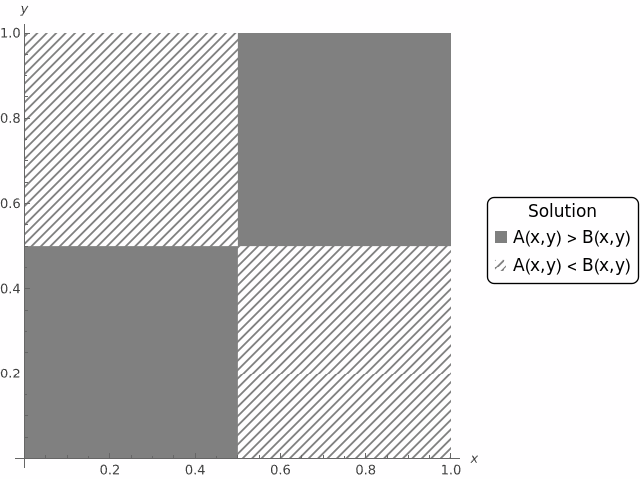}
        \captionof{figure}{$B(x,y)$ is sharper in the gray region, whereas $A(x,y)$ is sharper in the striped region.}
        \label{fig:image1}
    \end{minipage}

    \columnbreak
   
    Mathematica Code:
    \begin{verbatim}
    RegionPlot[{(x-1/2)(y-1/2) > 0,
    (x-1/2)(y-1/2) < 0}, {x, 0, 1},
    {y, 0, 1}, PlotStyle -> {Gray, 
    Directive[Gray, HatchFilling
    [Pi/4, 1.5]]}, BoundaryStyle 
    -> None, PlotLegends -> 
    SwatchLegend[Automatic, {"A(x,y)
    >B(x,y)", "A(x,y)<B(x,y)"}, 
    LegendFunction -> "Frame",
    LegendLabel -> "Solution"], 
    Frame -> False, Axes -> True,
    AxesLabel -> Automatic]
    \end{verbatim}
\end{multicols}

We have the following reduction of the above result.
\begin{corollary}\label{cor:diagonal}
For $0<x<1$,
\begin{equation}\label{eq:diagonal}
  \sin(\pi x) \;\leq\; \sin\!\bigl(2\pi x(1-x)\bigr),
\end{equation}
with equality if and only if $x=\tfrac12$.
\end{corollary}

\begin{proof}
Set $y=x$ in Corollary~\ref{cor:main12}.  The left-hand side becomes
\[
  \bigl(\sin(\pi x)\bigr)^x\,\bigl(\sin(\pi x)\bigr)^{1-x}
  = \sin(\pi x).
\]
The two bounds reduce to
\[
  A(x,x) = \Bigl(\sin\tfrac{\pi x}{1}\Bigr)^{\!1} = \sin(\pi x),
  \qquad
  B(x,x) = \Bigl(\sin\tfrac{2\pi x(1-x)}{1}\Bigr)^{\!1}
           = \sin\!\bigl(2\pi x(1-x)\bigr),
\]
so Corollary~\ref{cor:main12} gives
$\sin(\pi x)\leq\min\{\sin(\pi x),\,\sin(2\pi x(1-x))\}$.
By Corollary~\ref{cor:compare2D} with $y=x$, we have $B\leq A$ if and only if
$(x-\tfrac12)^2\geq 0$, which holds for all $x$, with equality only
at $x=\tfrac12$.  Hence $\min\{A,B\}=B$ for $x\neq\tfrac12$, giving
\eqref{eq:diagonal}; at $x=\tfrac12$ both sides equal $1$.

\end{proof}

\subsection{The three-dimensional case}\label{sec:3D}

We now restrict Theorem~\ref{thm:compare_n} to the case $n=3$.
The following is obtained by combining
Propositions~\ref{prop:main3} and~\ref{prop:main4} with $n=3$.

\begin{corollary}\label{cor:main34}
For $0<x,y,z<1$, define
\begin{align*}
  S(x,y,z)
  &:=\Bigl(\sin\tfrac\pi{x+y+z}\bigl[y(1-x)+z(1-y)+xz\bigr]
    \Bigr)^{\!x+y+z},\\
  T(x,y,z)
  &:=\Bigl(\sin\tfrac\pi{x+y+z}\bigl[xy+yz+xz\bigr]
    \Bigr)^{\!x+y+z}.
\end{align*}
Then
\begin{equation}\label{eq:main34}
  \bigl(\sin(\pi x)\bigr)^y\,
  \bigl(\sin(\pi y)\bigr)^z\,
  \bigl(\sin(\pi z)\bigr)^x
  \;\leq\;\min\{S(x,y,z),\,T(x,y,z)\}.
\end{equation}
\end{corollary}

The comparison between $S$ and $T$ is the instance $n=3$ of
Theorem~\ref{thm:compare_n}.  We record the explicit statement.

\begin{corollary}\label{cor:compare3D}
With $S$ and $T$ as in Corollary~\ref{cor:main34}, set
\[
  \delta := y(1-2x)+z(1-2y).
\]
Then
\begin{equation}\label{eq:3Dcond2}
  S(x,y,z)\;\geq\; T(x,y,z)
  \quad\Longleftrightarrow\quad
  \delta\,(2z-1)\;\leq\;0,
\end{equation}
with equality $S=T$ if and only if $\delta=0$ or $z=\tfrac12$.
More explicitly,
\begin{equation}\label{eq:3Dcond}
  \begin{cases}
    S > T & \text{if } z<\tfrac12 \text{ and } \delta>0, \\[4pt]
    S < T & \text{if } z<\tfrac12 \text{ and } \delta<0, \\[4pt]
    S = T & \text{if } z=\tfrac12 \text{ or } \delta=0, \\[4pt]
    S < T & \text{if } z>\tfrac12 \text{ and } \delta>0, \\[4pt]
    S > T & \text{if } z>\tfrac12 \text{ and } \delta<0.
  \end{cases}
\end{equation}
\end{corollary}

\begin{proof}
Apply Theorem~\ref{thm:compare_n} with $n=3$, $\lambda_1=x$,
$\lambda_2=y$, $\lambda_3=z$.  Then
\[
  \delta_3 = \lambda_2(1-2\lambda_1)+\lambda_3(1-2\lambda_2)
           = y(1-2x)+z(1-2y) = \delta,
\]
and the criterion $\delta_3(2\lambda_3-1)\leq 0$ becomes
$\delta(2z-1)\leq 0$, which is \eqref{eq:3Dcond2}.  The case
analysis of \eqref{eq:3Dcond} follows immediately from the signs of
$\delta$ and $(2z-1)$.

The equality $S=T$ holds if and only if $\delta_3=0$ (i.e., \ $\delta=0$)
or $\lambda_3=\tfrac12$ (i.e., \ $z=\tfrac12$).
\end{proof}

\begin{multicols}{2} 
    \noindent 
    \begin{minipage}{\columnwidth}
        \centering
        \includegraphics[width=\linewidth]{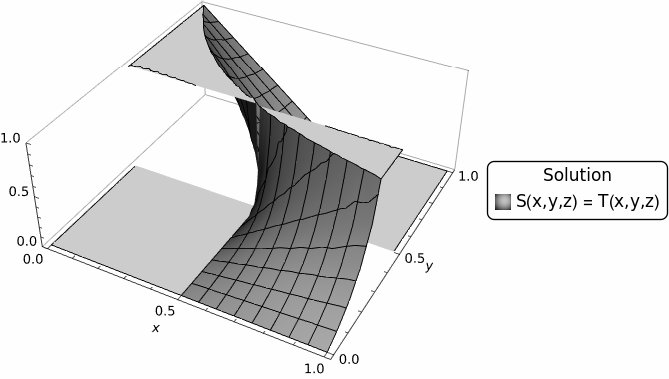} 
        \captionof{figure}{The condition $\delta = 0$, together with the critical point $z=\frac{1}{2}$, defines the boundary separating regions where each expression provides the sharper upper bound.}
        \label{fig:image2}
    \end{minipage}

    \columnbreak
    
    Mathematica Code:
    \begin{verbatim}
    Plot3D[(y (2 x - 1))/(1 - 2 y),
    {x, 0, 1}, {y, 0, 1}, PlotRange
    ->{0, 1}, PlotTheme->"Monochrome",
    AxesLabel->Automatic, PlotLegends
    -> SwatchLegend[{"S(x,y,z) = 
    T(x,y,z)"}, LegendFunction -> 
    "Frame",LegendLabel->"Solution"]]
    \end{verbatim}
\end{multicols}

\newpage
\begin{multicols}{2} 
    \noindent 
    \begin{minipage}{\columnwidth}
        \centering
        \includegraphics[width=\linewidth]{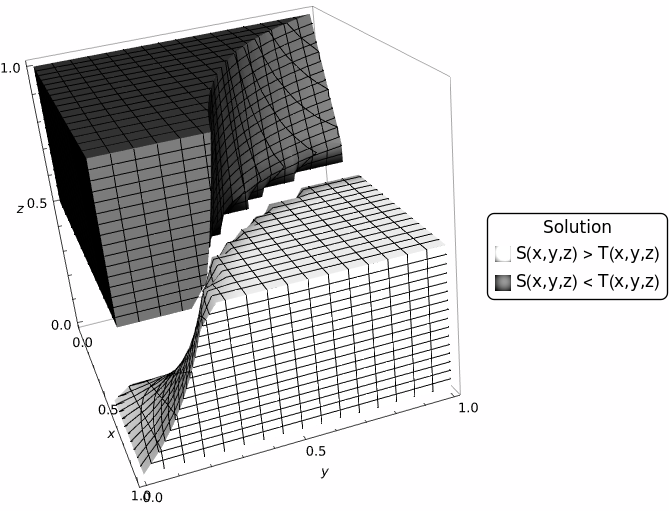} 
        \captionof{figure}{For $\delta>0$, the domain splits at $z=\frac{1}{2}$. In the light gray region $(z<\frac{1}{2})$, $T(x,y,z)$ provides the sharper upper bound, while in the dark gray region $(z>\frac{1}{2})$ $S(x,y,z)$ is sharper.}
        \label{fig:image3}
    \end{minipage}

    \columnbreak

    Mathematica Code:
    \begin{verbatim}
    RegionPlot3D[{(y (1 - 2 x) + 
    z (1 - 2 y))(2 z - 1)>0 && 
    z<1/2, (y (1 - 2 x) + z (1 - 2 y))
    (2 z - 1) > 0 && z > 1/2}, {x, 
    0, 1}, {y, 0, 1}, {z, 0, 1}, 
    PlotTheme -> "Monochrome", 
    PlotStyle -> {Gray, Black},
    BoundaryStyle -> None, AxesLabel
    -> Automatic, PlotLegends -> 
    SwatchLegend[Automatic, {"S(x,y,z)
    > T(x,y,z)", "S(x,y,z)<T(x,y,z)"},
    LegendFunction -> "Frame",
    LegendLabel -> "Solution"]]
    \end{verbatim}
\end{multicols}

\begin{multicols}{2} 
    \noindent 
    \begin{minipage}{\columnwidth}
        \centering
        \includegraphics[width=\linewidth]{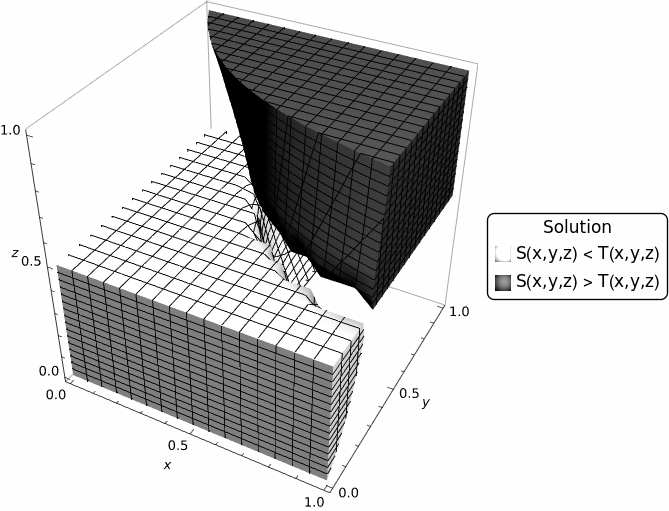}
        \captionof{figure}{For $\delta>0$, the domain splits at $z=\frac{1}{2}$. In the light gray region $(z<\frac{1}{2})$, $S(x,y,z)$ provides the sharper upper bound, while in the dark gray region $(z>\frac{1}{2})$, $T(x,y,z)$ is sharper.}
        \label{fig:image4}
    \end{minipage}

    \columnbreak
   
    Mathematica Code:
    \begin{verbatim}
    RegionPlot3D[{(y (1 - 2 x) +
    z (1 - 2 y))(2 z - 1)<0 && 
    z<1/2, (y (1 - 2 x) + z (1 - 2 y))
    (2 z - 1) < 0 && z > 1/2}, {x,
    0, 1}, {y, 0, 1}, {z, 0, 1}, 
    PlotTheme -> "Monochrome", 
    PlotStyle -> {Gray, Black},
    BoundaryStyle -> None, AxesLabel
    -> Automatic, PlotLegends ->
    SwatchLegend[Automatic, {"S(x,y,z)
    < T(x,y,z)", "S(x,y,z)>T(x,y,z)"},
    LegendFunction -> "Frame",
    LegendLabel -> "Solution"]]
    \end{verbatim}
\end{multicols}

We have the following reduction to the two angles from the above three angle results.
\begin{corollary}\label{cor:diagonal3D}
For $0<x<1$ and $0<z<1$,
\begin{equation}\label{eq:diagonal3D}
  \bigl(\sin(\pi x)\bigr)^{x+z}\,\bigl(\sin(\pi z)\bigr)^x
  \;\leq\;
  \begin{cases}
    \Bigl(\sin\left(\dfrac{\pi\left[x(1-x)+z\right]}{2x+z}\right)\Bigr)^{\!2x+z}
    & \text{if } (x-\tfrac12)(z-\tfrac12)\geq 0, \\[10pt]
    \Bigl(\sin\left(\dfrac{\pi x(x+2z)}{2x+z}\right)\Bigr)^{\!2x+z}
    & \text{if } (x-\tfrac12)(z-\tfrac12)\leq 0,
  \end{cases}
\end{equation}
with equality between the two right-hand sides when $x=\tfrac12$ or $z=\tfrac12$.
\end{corollary}

\begin{proof}
Set $y=x$ in Corollary~\ref{cor:main34} or consider Corollary \ref{cor:main34} for $n=2$ (take $\lambda_2=x$ and $\lambda_1=z$).  The left-hand side becomes
$(\sin(\pi x))^{x+z}(\sin(\pi z))^x$, and the two bounds are now
\begin{align*}
  S(x,x,z)
  &= \Bigl(\sin\tfrac{\pi\left[x(1-x)+z\right]}{2x+z}\Bigr)^{\!2x+z},\\[4pt]
  T(x,x,z)
  &= \Bigl(\sin\tfrac{\pi x(x+2z)}{2x+z}\Bigr)^{\!2x+z}.
\end{align*}
With $y=x$, the quantity $\delta$ of Corollary~\ref{cor:compare3D} becomes
\[
  \delta = x(1-2x)+z(1-2x) = (x+z)(1-2x).
\]
Since $x+z>0$, the criterion $\delta(2z-1)\leq 0$ reduces to
$(1-2x)(2z-1)\leq 0$, that is
\[
  (x-\tfrac12)(z-\tfrac12)\;\geq\;0.
\]
Thus $S\leq T$ (hence $\min\{S,T\}=S$) when $(x-\tfrac12)(z-\tfrac12)\geq 0$,
and $S\geq T$ (so $\min\{S,T\}=T$) when $(x-\tfrac12)(z-\tfrac12)\leq 0$.
The equality $S=T$ holds when $x=\tfrac12$ or $z=\tfrac12$.
\end{proof}

\begin{multicols}{2} 
    \noindent 
    \begin{minipage}{\columnwidth}
        \centering
        \includegraphics[width=\linewidth]{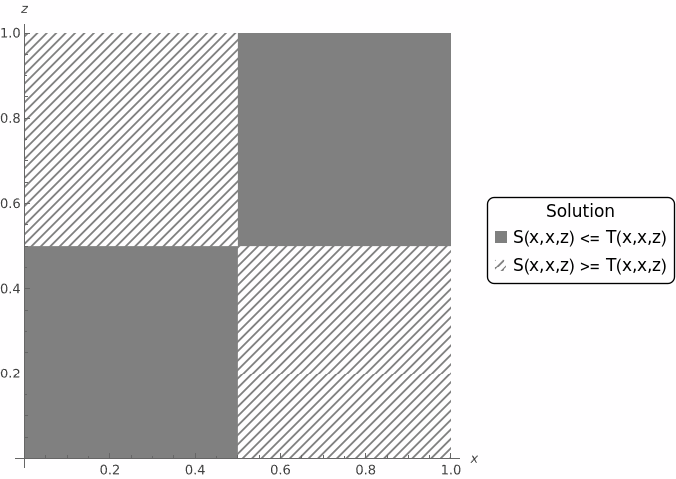}
        \captionof{figure}{$S(x,x,z)$ is sharper in the gray region, whereas $T(x,x,z)$ is sharper in the striped region.}
        \label{fig:image5}
    \end{minipage}

    \columnbreak
   
    Mathematica Code:
    \begin{verbatim}
    RegionPlot[{(x - 1/2)(z - 1/2)
    >= 0, (x - 1/2)(z - 1/2) <= 0},
    {x, 0, 1}, {z, 0, 1}, PlotStyle
    ->{Gray, Directive[Gray, 
    HatchFilling[Pi/4, 1.5]]},
    BoundaryStyle->None, PlotLegends
    -> SwatchLegend[Automatic, 
    {"S(x,x,z)<=T(x,x,z)", "S(x,x,z)
    >= T(x,x,z)"}, LegendFunction ->
    "Frame", LegendLabel->"Solution"],
    Frame -> False, Axes -> True,
    AxesLabel -> Automatic]
    \end{verbatim}
\end{multicols}

The following two angle result also follows.
\begin{corollary}\label{cor:diagonal3D2}
For $0<x<1$ and $0<y<1$,
\begin{equation}\label{eq:diagonal3D2}
  \bigl(\sin(\pi x)\bigr)^y\,\bigl(\sin(\pi y)\bigr)^{x+y}
  \;\leq\;
  \begin{cases}
    \Bigl(\sin\dfrac{\pi y(2-y)}{x+2y}\Bigr)^{\!x+2y}
    & \text{if } (x+y-1)(2y-1)\geq 0, \\[10pt]
    \Bigl(\sin\dfrac{\pi y(2x+y)}{x+2y}\Bigr)^{\!x+2y}
    & \text{if } (x+y-1)(2y-1)\leq 0,
  \end{cases}
\end{equation}
with equality between the two right-hand sides when $y=\tfrac12$ or $x+y=1$.
\end{corollary}

\begin{proof}
Set $z=y$ in Corollary~\ref{cor:main34} or consider Corollary \ref{cor:main34} for $n=2$ (take $\lambda_2=y$ and $\lambda_1=x$).  The left-hand side becomes
$(\sin(\pi x))^y(\sin(\pi y))^{x+y}$, and the two bounds become
\begin{align*}
  S(x,y,y)
  &= \Bigl(\sin\tfrac{\pi y(2-y)}{x+2y}\Bigr)^{\!x+2y},\\[4pt]
  T(x,y,y)
  &= \Bigl(\sin\tfrac{\pi y(2x+y)}{x+2y}\Bigr)^{\!x+2y}.
\end{align*}
With $z=y$, the quantity $\delta$ of Corollary~\ref{cor:compare3D} becomes
\[
  \delta = y(1-2x)+y(1-2y) = 2y(1-x-y),
\]
and the criterion $\delta(2z-1)\leq 0$ with $z=y$ becomes
$2y(1-x-y)(2y-1)\leq 0$.  Since $y>0$, this reduces to
\[
  (x+y-1)(2y-1)\;\geq\;0.
\]
Thus $S\leq T$ (so $\min\{S,T\}=S$) when $(x+y-1)(2y-1)\geq 0$,
i.e.\ when $x+y\geq 1$ and $y\geq\tfrac12$, or $x+y\leq 1$ and $y\leq\tfrac12$.
And $S\geq T$ (so $\min\{S,T\}=T$) in the complementary region.
Equality holds when $y=\tfrac12$ or $x+y=1$.
\end{proof}

\begin{multicols}{2} 
    \noindent 
    \begin{minipage}{\columnwidth}
        \centering
        \includegraphics[width=\linewidth]{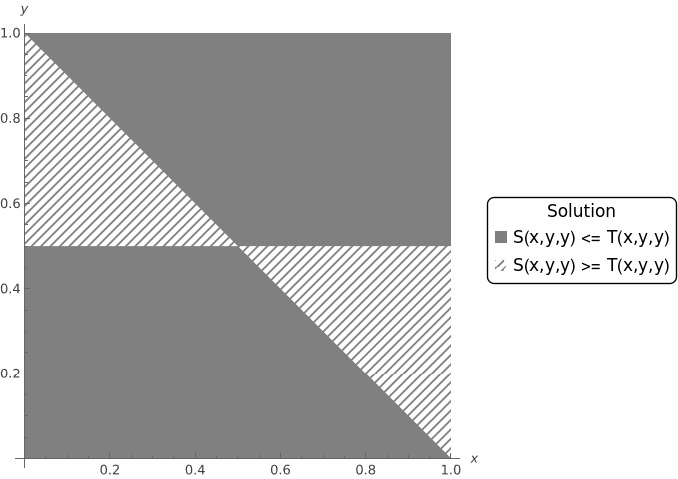} 
        \captionof{figure}{$S(x,y,y)$ is sharper in the gray region, whereas $T(x,y,y)$ is sharper in the striped region.}
        \label{fig:image6}
    \end{minipage}

    \columnbreak
   
    Mathematica Code:
    \begin{verbatim}
    RegionPlot[{(x + y - 1)(2 y - 1)
    >=0, (x + y - 1)(2 y - 1) <= 0},
    {x, 0, 1}, {y, 0, 1}, PlotStyle 
    -> {Gray, Directive[Gray, 
    HatchFilling[Pi/4, 1.5]]}, 
    BoundaryStyle->None, PlotLegends
    -> SwatchLegend[Automatic, 
    {"S(x,y,y)<=T(x,y,y)", "S(x,y,y)
    >= T(x,y,y)"}, LegendFunction ->
    "Frame", LegendLabel->"Solution"],
    Frame -> False, Axes -> True,
    AxesLabel -> Automatic]
    \end{verbatim}
\end{multicols}

The third two angle restriction result is as follows.
\begin{corollary}\label{cor:diagonal3D3}
For $0<x<1$ and $0<y<1$,
\begin{equation}\label{eq:diagonal3D3}
  \bigl(\sin(\pi x)\bigr)^{x+y}\,\bigl(\sin(\pi y)\bigr)^x
  \;\leq\;
  \begin{cases}
    \Bigl(\sin\dfrac{\pi [x(1+x)+y(1-2x)]}{2x+y}\Bigr)^{\!2x+y}
    & \text{if } (4xy-x-y)(2x-1)\geq 0, \\[10pt]
    \Bigl(\sin\dfrac{\pi x(x+2y)}{2x+y}\Bigr)^{\!2x+y}
    & \text{if } (4xy-x-y)(2x-1)\leq 0,
  \end{cases}
\end{equation}
with equality between the two right-hand sides when $x=\tfrac12$ or $4xy = x+y$.
\end{corollary}

\begin{proof}
Set $z=x$ in Corollary~\ref{cor:main34}.  The left-hand side becomes
$(\sin(\pi x))^{x+y}(\sin(\pi y))^y$, and the two bounds become
\begin{align*}
  S(x,y,x)
  &= \Bigl(\sin\dfrac{\pi [x(1+x)+y(1-2x)]}{2x+y}\Bigr)^{\!2x+y},\\[4pt]
  T(x,y,x)
  &= \Bigl(\sin\dfrac{\pi x(x+2y)}{2x+y}\Bigr)^{\!2x+y}.
\end{align*}
With $z=x$, the quantity $\delta$ of Corollary~\ref{cor:compare3D} becomes
\[
  \delta = y(1-2x)+x(1-2y) = x+y-4xy,
\]
and the criterion $\delta(2z-1)\leq 0$ with $z=x$ becomes
\[
  (4xy-x-y)(2x-1)\;\geq\;0.
\]
 Thus $S\leq T$ (so $\min\{S,T\}=S$) when $(4xy-x-y)(2x-1)\geq 0$,
 i.e.\ when $4xy \geq x+y$ and $x\geq\tfrac12$, or $4xy \leq x+y$ and $x\leq\tfrac12$.
 And $S\geq T$ (so $\min\{S,T\}=T$) in the complementary region.
 Equality holds when $x=\tfrac12$ or $4xy = x+y$.
\end{proof}

\begin{multicols}{2} 
    \noindent 
    \begin{minipage}{\columnwidth}
        \centering
        \includegraphics[width=\linewidth]{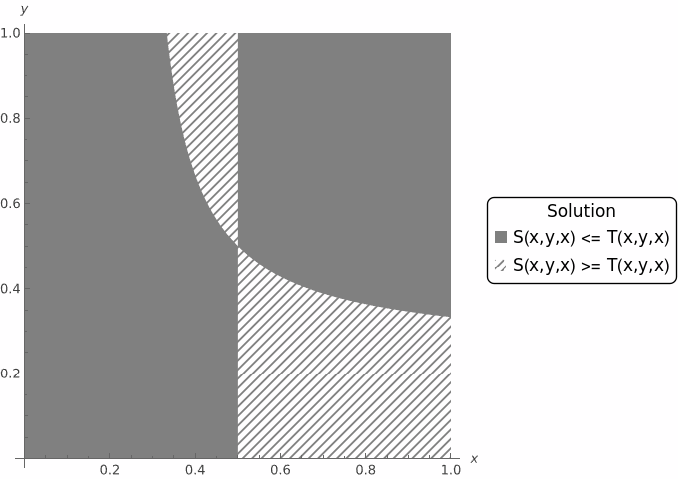}
        \captionof{figure}{$S(x,y,x)$ is sharper in the gray region, whereas $T(x,y,x)$ is sharper in the striped region.}
        \label{fig:image7}
    \end{minipage}

    \columnbreak
   
    Mathematica Code:
    \begin{verbatim}
    RegionPlot[{(x + y - 1)(2 y - 1)
    RegionPlot[{(4 x y - x - y)
    (2 x - 1) >= 0, (4 x y - x - y) 
    (2 x - 1) <= 0}, {x, 0, 1}, 
    {y, 0, 1}, PlotStyle -> {Gray,
    Directive[Gray, HatchFilling
    [Pi/4, 1.5]]}, BoundaryStyle ->
    None, PlotLegends->SwatchLegend[
    Automatic, {"S(x,y,x)<=T(x,y,x)",
    "S(x,y,x) >= T(x,y,x)"}, 
    LegendFunction -> "Frame", 
    LegendLabel -> "Solution"], 
    Frame -> False, Axes -> True, 
    AxesLabel -> Automatic]
    \end{verbatim}
\end{multicols}

Finally, combining (\ref{eq:ysinineq}), Corollary \ref{cor:diagonal3D}, Corollary \ref{cor:diagonal3D2}, and Corollary \ref{cor:diagonal3D3}, we obtain the following.
\begin{corollary}\label{cor:3Dto1D}
 Let $0<x<1$. Then
 \[
 \sin(\pi x)\leq \begin{cases}
    \sin\left(2\pi x(1-x)\right)
    & \text{if } x\in(0,\frac 13], \\[10pt]
    \sin\left(\frac{\pi}{3}(1+2x)\right)
    & \text{if } x\in[\frac 13,1).
  \end{cases}\]
\end{corollary}
\begin{proof}
Take $x=z$ in Corollary \ref{cor:diagonal3D}
and $x=y$ in Corollary \ref{cor:diagonal3D2} or Corollary \ref{cor:diagonal3D3}. We obtain
$$\sin(\pi x)\leq \min\left\{\sin(2\pi x(1-x)),\sin\left(\frac{2\pi}{3}(1-x)\right),\sin\left(\frac{\pi}{3}(2-x)\right)\right\}.$$
We observe that $\sin\left(\frac{2\pi}{3}(1-x)\right)=\sin\left(\frac{\pi}{3}(1+2x)\right)$ and $\sin\left(\frac{\pi}{3}(2-x)\right)=\sin\left(\frac{\pi}{3}(1+x)\right)$. Thus   $$\sin(\pi x)\leq \min\left\{\sin(2\pi x(1-x)),\sin\left(\frac{\pi}{3}(1+x)\right),\sin\left(\frac{\pi}{3}(1+2x)\right)\right\}.$$
Using Lemma \ref{lem:sincompare}, we obtain
\[
 \min\left\{\sin(2\pi x(1-x)),\sin\left(\frac{\pi}{3}(1+x)\right)\right\}= \begin{cases}
    \sin\left(2\pi x(1-x)\right)
    & \text{if } x\in(0,\frac 13]\cup [\frac 23,1), \\[10pt]
    \sin\left(\frac{\pi}{3}(1+x)\right)
    & \text{if } x\in[\frac 13,\frac 23];
  \end{cases}\]
  \[
 \min\left\{\sin(2\pi x(1-x)),\sin\left(\frac{\pi}{3}(1+2x)\right)\right\}= \begin{cases}
    \sin\left(2\pi x(1-x)\right)
    & \text{if } x\in(0,\frac 13], \\[10pt]
    \sin\left(\frac{\pi}{3}(1+2x)\right)
    & \text{if } x\in[\frac 13,1)
  \end{cases}\]
  and
  \[
 \min\left\{\sin\left(\frac{\pi}{3}(1+x)\right),\sin\left(\frac{\pi}{3}(1+2x)\right)\right\}= \begin{cases}
    \sin\left(\frac{\pi}{3}(1+x)\right)
    & \text{if } x\in(0,\frac 13], \\[10pt]
    \sin\left(\frac{\pi}{3}(1+2x)\right)
    & \text{if } x\in[\frac 13,1).
  \end{cases}\]
  We deduce that
   \[
 \min\left\{\sin(2\pi x(1-x)),\sin\left(\frac{\pi}{3}(1+x)\right),\sin\left(\frac{\pi}{3}(1+2x)\right)\right\}= \begin{cases}
    \sin\left(2\pi x(1-x)\right)
    & \text{if } x\in(0,\frac 13], \\[10pt]
    \sin\left(\frac{\pi}{3}(1+2x)\right)
    & \text{if } x\in[\frac 13,\frac 23], \\[10pt]
    \sin\left(\frac{\pi}{3}(1+2x)\right)
    & \text{if } x\in[\frac 23,1).
  \end{cases}\]
  The proof is complete.
\end{proof}

\begin{multicols}{2} 
    \noindent 
    \begin{minipage}{\columnwidth}
        \centering
        \includegraphics[width=\linewidth]{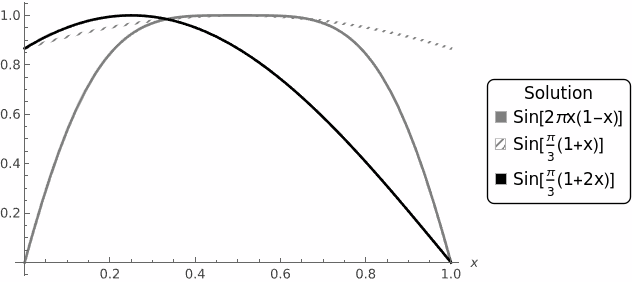} 
        \captionof{figure}{The gray curve provides the sharper bound on the interval $x\in (0,\frac{1}{3}]$. Likewise, the black curve yields the sharper bound on the interval $x\in [\frac{1}{3},1)$.}
        \label{fig:image8}
    \end{minipage}

    \columnbreak
   
    Mathematica Code:
    \begin{verbatim}
    Plot[{ConditionalExpression[Sin[
    2\[Pi] x (1 - x)], 0 < x <= 1],
    ConditionalExpression[Sin[\[Pi]/3
    (1 + x)], 0 <= x <= 1], 
    ConditionalExpression[Sin[\[Pi]/3
    (1 + 2 x)], 0 <= x < 1]}, {x, 0,
    1}, PlotStyle -> {Gray, Directive[
    Gray, HatchFilling[Pi/4, 1.5], 
    Thick], Black}, PlotLegends ->
    SwatchLegend[Automatic, {"Sin[2
    \[Pi]x(1-x)]", "Sin[\!\(
    \*FractionBox[\(\[Pi]\), \(3\)]\)
    (1+x)]", "Sin[\!\(\*FractionBox[
    \(\[Pi]\), \(3\)]\)(1+2x)]"}, 
    LegendFunction -> "Frame", 
    LegendLabel -> "Solution"], Frame 
    -> False Axes -> True, AxesLabel 
    -> Automatic]
    \end{verbatim}
\end{multicols}

\section{Conclusion and open problems}\label{sec:conclusion}

We have developed a unified Gamma-function approach to weighted
product inequalities for the sine function, and carried out a
complete analysis of the comparison between the two distinct upper
bounds that the approach produces.  
The two bounds $S$ and $T$ (or $P$ and $Q$) coincide exactly on
hyperplane sections of the parameter domain, and Theorems~\ref{thm:compare_n}
and~\ref{thm:compare_2n} identify these sections precisely.  A finer
question is how large the gap $|S-T|$ can be, and whether our approach produces bounds that are sharp in the sense
that equality in \eqref{eq:prodineqsine} is attained in the limit
as some parameters approach the boundary of $(0,1)^n$.

The log-convexity of the $q$-Gamma function $\Gamma_q$ for $0<q<1$,
established by Ismail, Lorch, and Muldoon \cite{Mitrinovic1970},
suggests that the entire framework of
Sections~\ref{sec:main}--\ref{sec:compare} could be carried out with
$\Gamma_q$ in place of $\Gamma$.  The corresponding reflection
formula involves the $q$-sine function, and the resulting
inequalities would be $q$-analogues of those proved here.


\end{document}